\documentclass{trans2-l}

\usepackage{enumitem}
\usepackage{amssymb}

\DeclareMathOperator\Tr{Tr}
\newtheorem{theorem}{Theorem}[section]

\newtheorem{lemma}[theorem]{Lemma}
\newtheorem{proposition}[theorem]{Proposition}

\theoremstyle{definition}
\newtheorem{definition}[theorem]{Definition}

\theoremstyle{remark}
\newtheorem{remark}[theorem]{Remark}

\numberwithin{equation}{section}

\begin{document}
\title[The CLT for the LES of the sum of independent rank one matrices]
{The Central Limit Theorem for  Linear Eigenvalue Statistics
 of the Sum of Independent Matrices of Rank One}

\author[O.\ Gu\'edon]{Olivier Gu\'edon}
\address{Department of Mathematics, University Paris-Est, 5, Bd Descartes, Champs sur Marne,
77454,  Marne la Vall\'{e}e, C\'{e}dex 2, France}
\email{{olivier.guedon@univ-mlv.fr}}

\author[A.\ Lytova]{Anna Lytova}
\address{Department of Mathematics, B. Verkin Institute for Low Temperature Physics and Engineering
of the NASU, 47 Lenin Ave., Kharkiv 61103, Ukraine}
\email{{anna.lytova@gmail.com};{lytova@ilt.kharkov.ua}}

\author[A.\ Pajor]{Alain Pajor}
\address{Department of Mathematics, University Paris-Est, 5, Bd Descartes, Champs sur Marne,
77454,  Marne la Vall\'{e}e, C\'{e}dex 2, France}
\email{{ Alain.Pajor@univ-mlv.fr}}

\author[L.\ Pastur]{Leonid Pastur}
\address{Department of Mathematics, B. Verkin Institute for Low Temperature Physics and Engineering
of the NASU, 47 Lenin Ave., Kharkiv 61103, Ukraine}
\email{{pastur@ilt.kharkov.ua}}

\dedicatory{Dedicated to Professor V. A. Machenko on occasion of his
90th birthday}

\thanks{Research was partly supported by the Network of Mathematical Research
(France - Ukraine) 2013 -- 2015.}
\thanks{L. P. is grateful to the Bezout Research Foundation (Labex Bezout)
for the invitation to the University Marne la Vall\'ee (France),
where the final version of the paper was prepared. }


\begin{abstract}
We consider $n\times n$  random matrices $M_{n}=\sum_{\alpha
=1}^{m}{\tau _{\alpha }}\mathbf{y}_{\alpha }\otimes
\mathbf{y}_{\alpha }$, where $\tau _{\alpha }\in \mathbb{R}$,
$\{\mathbf{y}_{\alpha }\}_{\alpha =1}^{m}$ are i.i.d. isotropic
random vectors of $\mathbb{R}^n$ (see Definition \ref{d:isotropic}),
whose components are not necessarily independent. It was shown in
\cite{Pa-Pa:09} that if $m,n\rightarrow \infty$,  $m/n\rightarrow
c\in \lbrack 0,\infty )$, the Normalized Counting Measures of
$\{\tau _{\alpha }\}_{\alpha =1}^{m}$ converge weakly and
$\{\mathbf{y}_\alpha\}_{\alpha=1}^m$ are \textit{good} (see
Definition \ref{d:good}), then the Normalized Counting Measures of
eigenvalues of $M_{n}$ converge weakly in probability to a
non-random limit found in \cite{Ma-Pa:67}. In this paper we indicate
a subclass of good vectors, which we  call \textit{very good} (see
Definition \ref{d:very}) and for which the linear eigenvalue
statistics of the corresponding matrices converge in distribution to
the Gaussian law, i.e., the Central Limit Theorem is valid (see
Theorem \ref{t:clt})). An important example of good vectors, studied
in \cite{Pa-Pa:09} are the vectors with log-concave distribution
(see Definition \ref{d:isotropic}). We discuss the conditions for
them, guaranteing the validity of the Central Limit Theorem for
linear eigenvalue statistics of corresponding matrices.
\end{abstract}

\maketitle

\section{Introduction: Problem and Main Result}

Let $\{\mathbf{y}_{\alpha }\}_{\alpha =1}^{m}$ be i.i.d. random vectors of $%
\mathbb{R}^{n}$, and $\{\tau _{\alpha }\}_{\alpha =1}^{m}$ be a collection
of real numbers. Consider the $n\times n$ real symmetric random matrix
\begin{equation}
M_{n}=\sum_{\alpha =1}^{m}\tau _{\alpha }L_{\mathbf{y}_{\alpha }},
\label{Mnm}
\end{equation}%
where$\;L_\mathbf{y}=\mathbf{y}\otimes \mathbf{y}$ is the $n\times n$
rank-one matrix defined as $L_{\mathbf{y}}\mathbf{x}=(\mathbf{y},\mathbf{x})%
\mathbf{y},\;\forall \mathbf{x}\in \mathbb{R}^{n}$ 
and $(\;,\;)$ is the standard Euclidean scalar product in $\mathbb{R}^{n}$.

Denote $\{\lambda^{(n)}_{l}\}_{l=1}^n$ the eigenvalues of $M_{n}$
counting
their multiplicity and introduce their Normalized Counting Measure (NCM) $%
N_{n}$, setting for any $\Delta \subset \mathbb{R}$
\begin{equation}
N_{n}(\Delta )=\mathrm{Card}\{l\in \lbrack 1,n]:\lambda^{(n)} _{l}\in
\Delta \}/n.  \label{NCM}
\end{equation}
Likewise, define the NCM $\sigma _{m}$ of $\{\tau _{\alpha }\}_{\alpha
=1}^{m}$,
\begin{equation}
\sigma _{m}(\Delta )=\mathrm{Card}\{\alpha\in \lbrack
1,m]:\tau_\alpha\in\Delta\}/m,  \label{sigm}
\end{equation}
and assume that the sequence $\{\sigma _{m}\}_{m=1}^\infty$ converges
weakly:
\begin{equation}
\lim_{m\rightarrow \infty }\sigma _{m}=\sigma, \; \sigma (\mathbb{R})=1.
\label{sigma}
\end{equation}
It follows from the results of \cite{Ma-Pa:67} that if (\ref{sigma})
holds, the mixed moments up to the 4th order of the components of
$\{\mathbf{y}_{\alpha }\}_{\alpha =1}^{m}$ satisfy certain
conditions as $n \to \infty$ (valid, in particular, for the vectors
uniformly distributed over the unit sphere of $\mathbb{R}^{n}$ (or
$\mathbb{C} ^{n})$ and for vectors with independent i.i.d.
components), and
\begin{equation}
n\rightarrow \infty ,\;m\rightarrow \infty ,\;m/n\rightarrow c\in \lbrack
0,\infty ),  \label{nmc}
\end{equation}
then there exists a non-random measure $N$ of total mass 1 such that for any
interval $\Delta \subset \mathbb{R}$ we have the convergence in probability
\begin{equation}
\lim_{n\rightarrow \infty ,\;m\rightarrow \infty ,\;m/n\rightarrow
c}N_{n}(\Delta )=N(\Delta ).  \label{NnmN}
\end{equation}
The measure $N$ can be found as follows. Introduce its Stieltjes transform
(see e.g. \cite{Ak-Gl:93})
\begin{equation}
f(z)=\int \frac{N(d\lambda )}{\lambda -z},\;\Im z\neq 0.  \label{f}
\end{equation}%
Here and below the integrals without limits denote the integrals over $%
\mathbb{R}$. Then $f$ is uniquely determined by the functional equation
\begin{equation}
zf(z)=c-1-c\int(1+\tau f(z))^{-1}\sigma(d\tau)  \label{MPE}
\end{equation}
considered in the class of functions analytic in $\mathbb{C\setminus \mathbb{%
\ R}}$ and such that $\Im f(z)\Im z\geq 0,\;\Im z\neq 0$. Since the
Stieltjes transform determines $N$ uniquely by the formula
\begin{equation}
\lim_{\varepsilon \rightarrow 0^{+}}\frac{1}{\pi }\int \varphi (\lambda )\Im
f(\lambda +i\varepsilon )d\lambda =\int \varphi (\lambda )N(d\lambda ),
\label{SP}
\end{equation}
valid for any continuous function with compact support, (\ref{MPE})
determines $N$ uniquely.

Note that  if the components $\{{y}_{\alpha j}\}_{j=1}^{n}$ of
$\mathbf{y}_{\alpha
},\;\alpha =1,...,m$ are i.i.d. random variables of zero mean and variance $%
1/n$, the matrix $M_{n}$ is usually written as
\begin{equation}  \label{SC}
M_{n}=Y^TTY, \quad Y= \{{y}_{\alpha k}\}^{m,n}_{\alpha,k=1},\quad
T=\{\tau_\alpha \delta_{\alpha\beta}\}_{\alpha,\beta=1}^m,
\end{equation}
and is closely related to the sample covariance matrix of statistics. A
particular case of this for $T=I_m$ and Gaussian $\{{y}_{\alpha
j}\}_{\alpha=j=1}^{m,n}$ is known since the 30s as the Wishart matrix (see e.g. \cite%
{Mui:82}).

The random matrices (\ref{Mnm}) appear also in the local theory
of Banach spaces and asymptotic convex geometry (see e.g. \cite%
{{Da-Sz:01},{Sz:06}}). A particular important case arising in this framework
is related to the study of geometric parameters associated to i.i.d. random
points
uniformly distributed over a convex body in $\mathbb{R}^n$ and of
the asymptotic geometry of random convex polytopes generated by these
points (see e.g. \cite{ {Au:07}, {Bou:96}, {Gi-Mi:00},
{Gi-Ts:03},{Li-Ru-Paj-Tom:05}}). This motivated to consider
random vectors known as isotropic and having a log-concave
distribution.
Recall the corresponding definitions.

\begin{definition}
\label{d:isotropic} (i) A random vector $\mathbf{y}=\{y_j\}_{j=1}^n
\in \mathbb{R}^{n}$ is called \textit{isotropic} if
\begin{equation}
\mathbf{E}\{y_{j} \}=0,\quad \mathbf{E}\{y_{j}y_{k}\}=n^{-1}\delta_{jk}, \;
j,k=1,...,n.  \label{isotropic}
\end{equation}

(ii) A measure $\mu $ on $\mathbb{C}^{n}$ is \textit{log-concave} if for any
measurable subsets $A,B$ of $\mathbb{C}^{n}$ and any $\theta \in \lbrack 0,1]
$,
\begin{equation*}
\mu (\theta A+(1-\theta )B)\geq \mu (A)^{\theta }\mu (B)^{(1-\theta )}
\end{equation*}
whenever $\theta A+(1-\theta )B=\{\theta y_{1}+(1-\theta )y_{2}\,:\,y_{1}\in
A,\;y_{2}\in B\}$ is measurable.
\end{definition}
It was proved in \cite{Pa-Pa:09} that (\ref{NnmN}) and (\ref{MPE})
remain valid in the case where the probability law of the i.i.d. vectors $\{%
\mathbf{y}_{\alpha }\}_{\alpha=1}^m$ is isotropic and log-concave.

In fact, a more general result was established in \cite{Pa-Pa:09}.
Introduce
\begin{definition}
\label{d:good} A random isotropic vector $\mathbf{y}\in \mathbb{R}^{n}$ is
called \textit{good} if for any $n\times n$ complex matrix $A_n$ which does
not depend on $\mathbf{y}$, we have
\begin{equation}
\mathbf{Var}\{(A_n\mathbf{y},\mathbf{y})\}\leq ||A_{n}||^2\delta_{n},\quad
\delta_{n}=o(1),\;n\rightarrow\infty,  \label{good}
\end{equation}
where $||A_{n}||$ is the operator norm of $A_{n}$.
\end{definition}
We have then \cite{Pa-Pa:09}:
\begin{theorem}
\label{t:MPP} Let $n$ and $m$ be positive integers satisfying (%
\ref{nmc}), $\{\mathbf{y}_{\alpha }\}_{\alpha =1}^{m}$ be i.i.d.
good vectors of \ $\mathbb{R}_n$, and $\{\tau _{\alpha }\}_{\alpha
=1}^{m}$ be real
numbers satisfying (\ref{sigma}). Consider the random matrix $M_{n}$ (\ref{Mnm}%
) and the Normalized Counting Measure of its eigenvalues $N_{n}$ (\ref{NCM}).
Then for any interval $\Delta \subset \mathbb{R}$ we have in probability
\begin{equation}
\lim_{n\rightarrow \infty ,m\rightarrow \infty ,m/n\rightarrow c\in \lbrack
0,\infty )}N_{n}(\Delta )=N(\Delta ),  \label{NnmNP}
\end{equation}
where the limiting non-random measure $N$ is given by (\ref{NnmN}) -- (\ref%
{SP}).
\end{theorem}


It follows from Theorem {\ref{t:MPP}} that if $M_{n}$ is given by (\ref%
{Mnm}), where $\{\mathbf{y}_{\alpha }\}_{\alpha =1}^{m_n}$ are i.i.d. good
vectors and
\begin{equation}
\mathcal{N}_n[\varphi]=\sum_{j=1}^n \varphi(\lambda^{(n)}_{j}),  \label{Nn}
\end{equation}
is the \emph{linear eigenvalue statistic} corresponding to any
continuous and bounded \emph{test-function} $\varphi: \mathbb{R} \to
\mathbb{C}$, then we have in probability
\begin{equation}
\lim_{n\rightarrow \infty ,m\rightarrow \infty ,m/n\rightarrow c\in \lbrack
0,\infty )}n^{-1}\mathcal{N}_n[\varphi]=\int \varphi(\lambda)dN(\lambda).
\label{Nnl}
\end{equation}
This can be viewed as an analog of the Law of Large Numbers of probability
theory for (\ref{Nn}). In this paper we deal with the fluctuations of $%
\mathcal{N}_n$ around its limit (\ref{Nn}), i.e. with an analog of the
Central Limit Theorem (CLT) of probability theory. Our goal is to find a
class of i.i.d. good vectors and a class of test functions such that
the centered and appropriately normalized linear eigenvalue statistics
\begin{equation}
\mathcal{N}^\circ_n[\varphi]/\nu_n, \quad \mathcal{N}^\circ_n[\varphi]=
(\mathcal{N}_n[\varphi]-\mathbf{E}\{\mathcal{N}%
_n[\varphi]\}  \label{Nn0}
\end{equation}
converge in distribution to a Gaussian random variable.

There is a number of papers on the CLT for linear eigenvalue
statistics of matrices (\ref{Mnm}) where $\{{y}_{\alpha
j}\}_{\alpha,j=1}^{m,n}$ are independent, i.e. for sample covariance
matrices (\ref{SC})
(see \cite{Ba-Si:04,Gi:01,Ly-Pa:08,P-Zh:08,Sh:11} and references therein).
Unfortunately, much less is known in the case where the components of $%
\mathbf{y}_\alpha$'s are dependent (see e.g. \cite{Pa-Sh:11}, Chapter 17 and
references therein).

An important step in proving the CLT is the asymptotic analysis of the
variance of the corresponding linear statistic
\begin{equation}
\mathbf{Var}\{\mathcal{N}_n[\varphi]\} := \mathbf{E}\{(\mathcal{N}%
^\circ_n[\varphi])^2\},  \label{vare}
\end{equation}
in particular, the proof of a bound
\begin{equation}
\mathbf{Var}\{\mathcal{N}_n[\varphi]\}\leq C_n||\varphi||^2_H,
\label{Nvar}
\end{equation}
where $||...||_H$ is a functional norm and $C_n$ depends only on
$n$, or even an asymptotic of the variance. This determines the
normalization factor $\nu_n$ in (\ref{Nn0}) and the class $H$ of
test-functions for which the CLT is valid.

It appears that for many random matrices normalized so that there
exists a limit of their NCM, in particular for sample covariance
matrices (\ref{SC}), the variance of linear eigenvalue statistic
with $\varphi \in C^1$ admits the bound
\begin{equation}
\mathbf{Var}\{\mathcal{N}_n[\varphi]\}=O(1),\quad
n\rightarrow\infty, \label{VarO1}
\end{equation}
or even a limit as $n \to \infty$. Thus the CLT has to be valid
for (\ref{Nn0}) without any $n$-dependent normalization factor $\nu_n$
\cite{Pa-Sh:11}. This has to be compared with the generic situation
in probability theory, where the variance of a linear statistic of
i.i.d. random variables is proportional to $n$ for any bounded
$\varphi$, hence the CLT is valid for an analog of (\ref{Nn0}) with
$\nu=n^{-1/2}$.

To formulate the version of (\ref{VarO1}), which we will prove and use in
this paper, introduce
\begin{definition}
\label{d:unc} The distribution of random vector $\mathbf{y}\in\mathbb{R}^n$
is called \textit{unconditional} if its components $\{y_j\}_{j=1}^n$ have
the same joint distribution as $\{\pm y_j\}_{j=1}^n$ for any choice of signs.
\end{definition}
\begin{lemma}
\label{l:Var} Let
$\mathbf{y}=\{y_{i}\}_{j=1}^{n}$ be an isotropic random
vector having an unconditional distribution and satisfying
\begin{equation}\label{a22L}
a_{2,2}:=\mathbf{E}\{y_{j}^{2}y_{k}^{2}\}=n^{-2}+O(n^{-3}),\;j\neq k,\;\;
\kappa _{4} :=\mathbf{E}\{y_{j}^{4}\}-3a_{2,2}=O(n^{-2}).
\end{equation}
Consider the random matrix $M_{n}$ of (\ref{Mnm}) in which $m$ and $n$ satisfy
(\ref{nmc}), $\{y_{\alpha
}\}_{\alpha =1}^{m}$ are i.i.d. random vectors satisfying (\ref{a22L}) and $%
\{\tau _{\alpha }\}_{\alpha =1}^{m}$  are non-negative real numbers with
the limiting counting distribution $\sigma $ (\ref{sigma}) having a finite
fourth moment:
\begin{equation}
m_{4}:=\int_{0}^{\infty }\tau ^{4}d\sigma (\tau )<\infty .  \label{m4}
\end{equation}%
Then we have for all sufficiently large $m$ and $n$
\begin{equation}
\mathbf{Var}\{\mathcal{N}_{n}[\varphi ]\}\leq C||\varphi ||_{2+\delta }^{2},
\label{apriory}
\end{equation}%
where $C$ is an absolute constant and
\begin{equation}
||\varphi ||_{2+\delta}^{2}=\int (1+2|k|)^{2(2+\delta)}|\widehat{\varphi }(k)|^{2}dk,\quad
\widehat{\varphi }(k)=\int e^{ikx}\varphi (x)dx.  \label{Hs}
\end{equation}
\end{lemma}

The proof of the lemma is given in Section \ref{s:var}. It turns out
however that the validity of the CLT, more exactly, its proof in
this paper, requires more conditions on the components of random
vectors $\{y_\alpha\}_{\alpha=1}^m$ in (\ref{Mnm}). Namely, we
introduce

\begin{definition}
\label{d:very} A random isotropic vector $\mathbf{y}=\{y_j\}_{j=1}^n \in\mathbb{%
R}^n$ is called \textit{very good} if its distribution is unconditional,
there exist $n$-independent $a,b\in\mathbb{R}$ such that (cf. (\ref{a22L}))
\begin{align}
& a_{2,2}:=\mathbf{E}\{y_{ j}^2y_{ k}^2\}=n^{-2}+an^{-3}+o(n^{-3}),\quad
j\neq k, n\rightarrow\infty, \label{a22} \\
&\kappa_{4}:=\mathbf{E}\{y_{ j}^4\}-3a_{2,2}=bn^{-2}+o(n^{-2}),\quad
n\rightarrow\infty,  \label{k4}
\end{align}
and
\begin{equation}
\mathbf{E}\{|(A_n\mathbf{y},\mathbf{y})^{\circ}|^4\}\leq ||A_{n}||^{4}%
\widetilde{\delta}_{n},\quad \widetilde{\delta}_{n}=O(n^{-2}),\;n\rightarrow%
\infty,  \label{m4Ayy}
\end{equation}
for any $n\times n$ complex matrix $A_n$ which does not depend on $\mathbf{y}
$.
\end{definition}

It is easy to check that the vectors
$\mathbf{y}=\mathbf{x}/n^{1/2}$, where $\mathbf{x}$ has i.i.d.
components with even distribution and such that
$\mathbf{E}\{x_j^8\}< \infty$, are very good as well as vectors
uniformly distributed over the unit ball or unit sphere of
$\mathbb{R}^n$. For other examples of very good random vectors of
geometric origin see Section \ref{s:vectors}.

\begin{remark}
Here instead of unconditionality one can assume that $\mathbf{y}$
satisfies condition (\ref{yyyy}) below (like it was assumed in \cite{Ma-Pa:67}).
\end{remark}
Now we are ready to formulate our main result:
\begin{theorem}
\label{t:clt} Let $m$ and $n$ be positive integers satisfying (%
\ref{nmc}), $\{\mathbf{y}_{\alpha }\}_{\alpha =1}^{m}$ be i.i.d.
very good vectors in the sense of Definition \ref{d:very} and
$\{\tau _{\alpha }\}_{\alpha =1}^{m}$ be
non-negative numbers satisfying (\ref{sigma}) and (\ref{m4}).
Consider the corresponding random matrix $M_{n}$ of (\ref{Mnm}) and
a linear statistic $\mathcal{N}_{n}[\varphi ]$ (\ref{Nn}) of its
eigenvalues. Assume that $\varphi:\mathbb{R}\rightarrow\mathbb{R}, \;\varphi \in H_{2+\delta }, \; \delta >0$ (see (\ref{Hs})).
 Then $\mathcal{N}_{n}^{\circ }[\varphi ]$ of (\ref{Nn0})
converges in distribution to the Gaussian random variable with zero
mean and variance $V[\varphi ]=\lim_{y\downarrow 0}V_{y}[\varphi
]$, where
\begin{align}
& V_{\eta }[\varphi ]=\frac{1}{2\pi ^{2}}\int \int \Re \big[C(z_{1},%
\overline{z_{2}})-C(z_{1},z_{2})\big]\varphi (\lambda _{1})
{\varphi (\lambda
_{2})}d\lambda _{1}d\lambda _{2}\;,  \label{Var} \\
& C(z_{1},z_{2})=\frac{\partial ^{2}}{\partial z_{1}\partial z_{2}}\bigg(%
2\log \frac{\Delta f}{\Delta z}-(a+b)f(z_{1})f(z_{2})\frac{\Delta z}{\Delta f}%
\bigg),  \label{Cov}
\end{align}%
with $z_{1,2}=\lambda _{1,2}+i\eta $, $\Delta f=f(z_{1})-f(z_{2})$, $\Delta
z=z_{1}-z_{2}$, and $f$ is given by (\ref{MPE}).
\end{theorem}
\begin{remark}
\label{r:covgg} Note that in fact
\begin{align*}
C(z_1,z_2)=\lim_{n\rightarrow\infty}\mathbf{E}\{\gamma_{n}^{\circ}(z_1)%
\gamma_{n}(z_2)\},
\end{align*}
where $\gamma_{n}(z)=\Tr (M_{n}-zI_n)^{-1}$ and
$\gamma_{n}^{\circ}=\gamma_{n}^{\circ}-\mathbf{E}\{\gamma_{n}\}$.
\end{remark}
\begin{remark}
The condition $\tau_\alpha\geq 0$, $\alpha=1,...,m$, is a pure technical
one, it can be shown that the results remain valid for $\tau_\alpha\in%
\mathbb{R}$.
\end{remark}
\begin{remark}
One can also rewrite $V_{\eta }[\varphi ]$ in the form
\begin{align}\label{Var2}
V_{\eta }[\varphi ]=&\frac{1}{2\pi ^{2}}\int \int \Re \big[L(z_{1},{z_{2}}%
)-L(z_{1}\overline{z_{2}})\big](\varphi (\lambda _{1})-\varphi (\lambda
_{2}))^{2}d\lambda _{1}d\lambda _{2}   \notag\\
\; &+ \frac{(a+b)c}{\pi ^{2}}\int_{0}^{\infty }\tau ^{2}d\sigma (\tau )\bigg(%
\Im \int \frac{f^{\prime }(z_{1})}{(1+\tau f(z_{1}))^{2}}\varphi
(\lambda _{1})d\lambda _{1}\bigg )^{2},
\end{align}
\begin{align}
L(z_{1},z_{2})=\frac{{\partial ^{2}}}{\partial z_{1}\partial
z_{2}}\log \frac{\Delta f}{\Delta z}. \notag
\end{align}%
In particular, if $\tau _{\alpha }=1$,
$\alpha =1,...,m$, then
\begin{align}
V[\varphi ]|_{\tau =1}=\frac{1}{2\pi ^{2}}\int_{a_{-}}^{a_{+}}%
\int_{a_{-}}^{a_{+}}& \left(\frac{\triangle \varphi }{\triangle \lambda }%
\right) ^{2}\frac{(4c-(\lambda _{1}-a_{m})(\lambda
_{2}-a_{m}))d\lambda
_{1}d\lambda _{2}}{\sqrt{(a_{+}-\lambda _{1})(\lambda _{1}-a_{-})}\sqrt{%
(a_{+}-\lambda _{2})(\lambda _{2}-a_{-})}}  \notag \\
& +\frac{a+b}{4c\pi ^{2}}\left( \int_{a_{-}}^{a_{+}}\varphi (\mu )\frac{\mu
-a_{m}}{\sqrt{(a_{+}-\mu )(\mu -a_{-})}}d\mu \right) ^{2},  \label{VSC}
\end{align}%
where $\triangle \varphi =\varphi (\lambda _{1})-\varphi (\lambda _{2})$, $%
a_{\pm }=(1\pm \sqrt{c})^{2}$, $a_{m}=1+c$. This expression coincides with
that one for the limiting variance of linear eigenvalue statistics of sample
covariance matrices (see \cite{Ly-Pa:08}), in which the fourth cumulant of
matrix entries is replaced with $a+b$.
\end{remark}

The paper is organized as follows. Section \ref{s:vectors} presents some
facts on the isotropic random vectors with a log-concave unconditional
symmetric distribution. In Section \ref{s:var} we prove Lemma \ref{l:Var}.
Section \ref{s:main} presents the proof of our main result, Theorem \ref%
{t:clt}. Section \ref{s:aux} contains auxiliary results.

\section{ Isotropic random vectors with log-concave distribution}
\label{s:vectors} Let $\mathbf{y}\in\mathbb{R}^n$ be a random
isotropic\ vector with a log-concave density (see Definition
\ref{d:isotropic}). A  typical example from convex geometry is
vector uniformly distributed over a convex body in $\mathbb{R}^n$.
The study of the concentration of the Euclidean norm of
$\mathbf{y}$ around its average is a part of an important branch of
high dimensional convex geometry related to a famous conjecture of
Kannan, Lov\'{a}sz and Simonovits \cite{KLS:95} (see also the
surveys \cite{G:13, Vemp:05}) on the validity of Poincar\'{e}
type inequality.

In particular, the so-called thin shell conjecture claims that
$$
\mathbf{P}\left\{\left| \ ||\mathbf{y}||-1\right|>t\right\}\leq 2
\exp(-ct\sqrt n)
$$
where $c>0$ is a universal constant. A weaker conjecture, known as
the variance conjecture, claims that
\begin{equation}
\label{equ:variance conject} \mathbf{Var}\{||\mathbf{y}||^2\} \le C/n,
\end{equation}
where $C$ is a universal constant. The conjecture in full generality
is still open.

The first breakthrough was obtained by \cite{FGP:06, Kl2:07} where
the bound
\begin{equation}\label{beta}
\mathbf{Var}\{||\mathbf{y}||^2\} = o(1), \; n \to \infty
\end{equation}
was proved. The bound is the basic tool to prove Berry-Ess\'{e}en
type inequalities for one-dimensional marginals of $\mathbf{y}$
\cite{ABP:03, Kl2:07, Bob:04} and is sometimes called CLT for convex
bodies. The best known improvement of (\ref{beta}) by now  is
\cite{GM:11}
$$\mathbf{Var}\{||\mathbf{y}||^2\}= O(n^{-1/3}),\  n\to\infty. $$

The variance conjecture  (\ref{equ:variance conject}) has been
proved in certain special cases. Anttila et al. \cite{ABP:03}
considered random isotropic vectors uniformly distributed over
 the unit ball $B_p^n$ of the $\ell_p^n$ norm in $\mathbb{R}^n$
 and Wojtaszczyk \cite{Woj:07} considered the same setting for
a generalized Orlicz unit ball.  Klartag \cite{Kl:09} studied
vectors with the log-concave unconditional isotropic distribution
(see Definitions \ref{d:isotropic} and \ref{d:unc}).

While in high dimensional convex geometry one focuses mainly on
{\em quantitative} estimates as above, in this  paper, we will also
need  precise asymptotics
for mixed moments of the components of $\mathbf{y}$.  This raises
new questions in high dimensional convex geometry, related to
general quadratic forms rather than for norms. More precisely, let
$A_{n}$ be a $n\times n$ complex matrix such that $\|A_{n}\|\le 1$.
It was proved in \cite{Pa-Pa:09} that
$$
\mathbf{Var}\{(A_n\mathbf{y},\mathbf{y})\}= o(1),\  n\to\infty.
$$
According to Lemma \ref{l:Var}, we need the best possible bound,
i.e., an analog of (\ref{equ:variance conject}) for quadratic forms.
It is for instance known when $\mathbf{y}$ is uniformly distributed
over the unit ball $B_p^n$ and follows from the corresponding
Poincar\'e type estimates \cite{LW:08, S:08}. We prove below in
Lemma \ref{l:vectors} that the analog (\ref{VarL}) is valid for any
random vector with a log-concave unconditional symmetric isotropic
distribution.

To prove the CLT (see Theorem \ref{t:clt}), we will need precise
asymptotics for the mixed moments of components of $\mathbf{y}$ (see
(\ref{a22L}) and (\ref{a22}) -- (\ref{k4})), as well as more bounds
for quadratic forms (see (\ref{m4Ayy})). Given the parameters
$a_{2,2}$ and $\kappa_4$, we are considering in fact a sequence of
$n$-dimensional log-concave isotropic distributions satisfying
(\ref{a22L}) or (\ref{a22}) --
 (\ref{k4}).  From a geometric point of view, one may consider a
sequence of $n$-dimensional convex bodies, such as unit balls of
norms, and their uniform distributions (normalized to be isotropic).
A natural example is given by the sequence of the unit balls
$\{B_p^n\}_{n \in \mathbb{N}}$ for which we check that (\ref{a22L})
and (\ref{a22}) -- (\ref{k4})) are valid with $a$ and $b$ depending
only on $p$ (see (\ref{abp})). As for the general case, we have
\begin{lemma}
\label{l:vectors} If $\mathbf{y}=(y_1,...,y_n)\in\mathbb{R}^n$ is an
isotropic vector with a log-concave unconditional symmetric distribution,
then it satisfies (\ref{a22L}) and (\ref{good}) with $\delta_n=O(n^{-1})$, $%
n\rightarrow\infty$, and (\ref{m4Ayy}).
\end{lemma}
\begin{proof}
We will use the dimension free Khinchine-Kahane-type inequality
by Bourgain \cite{Bou:91} (see also \cite{Bob:00}):  if $P_d$ is a
polynomial of degree $d$, and  $\mathbf{y}\in\mathbb{R}^n$  has a
log-concave distribution, then
\begin{equation}
\mathbf{E}\{|P_d(\mathbf{y})|^q\}\leq
C(d,q)\mathbf{E}\{|P_d(\mathbf{y})|\}^q,
\label{bou}
\end{equation}
where $C(d,q)$ depends only on $d$ and $q$ and does not depend on
$n$. By using (\ref{bou}) and (\ref{isotropic}), we obtain for the
fourth moments of coordinates $\{y_j\}_{j=1}^n$ of $\mathbf{y}$:
 \begin{equation}
  \mathbf{E}\{|y_j|^4\}\leq C
  \mathbf{E}\{y_j^2\}^{2}=C n^{-2}.
 \label{bor}
 \end{equation}
If the distribution of $\mathbf{y}$ is symmetric, then
 \begin{equation}
 a_{2,2}=\frac{1}{n-1}\mathbf{E}\Big\{\sum_{k=1}^ny_{ j}^2y_{ k}^2-y_{ j}^4\Big\}=
 \frac{1}{n(n-1)}\mathbf{E}\{||\mathbf{y}||^4\}-\frac{1}{n-1}\mathbf{E}\{y_{j}^4\}.
 \label{a22=}
 \end{equation}
It follows from (\ref{equ:variance conject}) and (\ref{isotropic})
that
  \begin{equation}
  \mathbf{E}\{||\mathbf{y}||^4\}=1+ \mathbf{Var}\{||\mathbf{y}||^2\} \leq 1+ C/n.
  \label{y4}
 \end{equation}
 This and (\ref{a22=}) yield
 \begin{equation}
  a_{2,2} \le n^{-2} + C/n^{3}.
  \label{a22<}
 \end{equation}
 On the other hand $\mathbf{E}\{||\mathbf{y}||^4\}\geq\mathbf{E}\{||\mathbf{y}||^2\}^{2}=1$, which together with  (\ref{bor}) and (\ref{a22=}) lead to
$a_{2,2}\geq n^{-2}+C'/n^{3}$,
 and we get the first part of  (\ref{a22L}). The second part follows from  the first one and (\ref{bor}).

Since for any random $\mathbf{y}=\{y_j\}_{j=1}^n$ with unconditional
distribution
\begin{equation}
 \mathbf{E}\{y_{j}y_{k}y_{p}y_{q}\}=
 a_{2,2}(\delta_{jk}\delta_{pq}+\delta_{jp}\delta_{kq}+\delta_{jq}
 \delta_{kp})+\kappa_{4}\delta_{jk}\delta_{jp}\delta_{jq},
\label{yyyy}
 \end{equation}
we have for a symmetric matrix $A_n$
 \begin{align}
 \mathbf{Var}\{(A_n\mathbf{y},\mathbf{y})\}
 =(a_{2,2}-n^{-2})|\Tr A_n|^2+2a_{ 2,2}\Tr |A_{n}|^2+\kappa_{4}\sum_{j=1}^n |A_{njj}|^2,\label{Var=}
 \end{align}
 where $|A_{n}|^2=A_{n}A_{n}^*$. This and
  (\ref{a22L})  lead to
  \begin{equation}
  \mathbf{Var}\{(A_n\mathbf{y},\mathbf{y})\} = O(n^{-1}).\label{VarL}
  \end{equation}
In addition, it follows from (\ref{bou})
\begin{equation}
\mathbf{E}\{|(A_n\mathbf{y},\mathbf{y})^\circ|^{4}\}\leq C\mathbf{Var}\{(A_n\mathbf{y},\mathbf{y})\}
^{2},
\label{A42}
\end{equation}
which together with (\ref{VarL}) yield (\ref{m4Ayy}).
\end{proof}
Note, however, that not too much is known on isotropic vectors
with a log-concave distribution,
which satisfy (\ref{a22}) -- (\ref{k4}), i.e.,
are very good in the sense of Definition \ref{d:very}.
Thus, it could happen that some of them do not satisfy (\ref{a22})
and/or (\ref{k4}), for instance the coefficient in front of
$n^{-3}$ and/or the coefficient in front of $n^{-2}$ would "oscillate" in $n$.
This would mean that different subsequences of
vectors can have different coefficients $a$ and $b$ in (\ref{a22}) -- (\ref%
{k4}). Correspondingly, different subsequences of characteristic
functions of $\mathcal{N}^\circ_n[\varphi]$ would have Gaussian
limits with different variances (\ref{Var}) -- (\ref{Cov}). This
situation, if it would be the case, could be compared with that of
\cite{Pa:06b}, where it was shown that the limiting forms of the
variance and the probability law of fluctuations of linear
eigenvalue statistics for certain unitary invariant matrix ensembles
depend on a sequence $n_j\rightarrow\infty$.

Here are the examples of very good vectors. It was mentioned in the
previous section that  the vectors $\mathbf{y}=n^{-1/2}\mathbf{x}$,
where components $\{x_j\}_{j=1}^n$ of  $\mathbf{x}$ are i.i.d.
random variables with an even distribution such that
$\mathbf{E}\{x_j^2\}=1$, $\mathbf{E}\{x_j^4\}=m_4$, $\mathbf{E}\{x_j^{8}\}<\infty$
are very good with $a_{2,2}=n^{-2}$ and $\kappa_{4}=n^{-2}(m_4-3)$, thus
$a=0$ and $b=m_4-3$ in (\ref{a22}) -- (\ref{k4}).
Note that in this case Lemma \ref%
{l:vectors} is valid without the assumption of log-concave distribution of $\mathbf{y}$.

It can also be shown that the vectors uniformly distributed over the
Euclidean unit ball in $\mathbb{R}^n$ are also very good. Let us consider a
more general case, where $\mathbf{x}$ is uniformly distributed over
the unit ball
\begin{equation*}
B_p^n=\big\{\mathbf{x}\in\mathbb{R}^n:\;||\mathbf{x}||_p=\big(\sum_{j=1}^n|x_j|^p\big)^{1/p}\leq
1\big\}
\end{equation*}
of the space $l_p^n$. According to \cite{BGMN:05}, we have
\begin{align*}
\mathbf{E}_{B_p^n}\{f\}:&=\frac{1}{|B_p^n|}\int_{B_p^n}f(\mathbf{x})d \mathbf{x} \\
&=\frac{1}{(2\Gamma(1+1/p))^{n}} \int_0^\infty dt e^{-t}\int_{\mathbb{R}^n}d%
\mathbf{x}e^{-||\mathbf{x}||^p_p}f(\mathbf{x}(||\mathbf{x}||^p_p+t)^{-1/p}).
\end{align*}
This allows us to calculate the moments of $\mathbf{x}$ and to show
that the vector
\begin{equation*}
\mathbf{y}=\left(\frac{1}{n}\frac{B(1/p,2/p)}{B(n/p+1,2/p)}\right)^{1/2}%
\mathbf{x,}
\end{equation*}
where $B$ is the $\beta$-function, is isotropic
and satisfies (\ref{a22}) -- (\ref{k4}) with
\begin{equation}\label{abp}
a=-\frac{8}{p},\quad
b=\frac{\Gamma(1/p)\Gamma(5/p)}{\Gamma(3/p)^2}-3.
\end{equation}

\section{Proof of Lemma \protect\ref{l:Var}}\label{s:var}
We will essentially follow the scheme proposed in \cite{Sh:11},
which is based on two main ingredients. The
first is an inequality that allows us to transform bounds for the variances of
the trace of resolvent of a random matrix into bounds for the variances
of linear eigenvalue statistics with a sufficiently smooth test function:
\begin{proposition}
\label{p:varNf} Let $M_{n}$ be an $n\times n$ random matrix and $\mathcal{N}%
_{n}[\varphi ]$ be a linear statistic of its eigenvalues (see (\ref{Nn})).
Then we have for any $s>0$
\begin{equation}
\mathbf{Var}\{\mathcal{N}_{n}[\varphi ]\}\leq C_{s}||\varphi
||_{s}^{2}\int_{0}^{\infty }d\eta e^{-\eta }\eta ^{2s-1}\int_{-\infty
}^{\infty }\mathbf{Var}\{\gamma _{n}(\lambda +i\eta )\}d\lambda ,
\label{varNf}
\end{equation}%
where $C_{s}$ depends only on $s$,  $||\varphi ||_{s}$ is defined in (\ref%
{Hs}), $G(z)=(M_{n}-z)^{-1}$ is the resolvent of $M_{n}$ and
\begin{equation}
\gamma _{n}(z)=\Tr G(z).  \label{gamma}
\end{equation}
\end{proposition}
The second ingredient of the scheme of  \cite{Sh:11} is an improved version
of the martingale approach, providing  the bound $\mathbf{Var}\{\gamma
_{n}(z)\}\leq C(z)$ instead of the bound $\mathbf{Var}\{\gamma _{n}(z)\}\leq
C(z)n$ \ (see e.g. \cite{Pa-Sh:11}, Theorem 19.1.6). Using this version, we
prove
\begin{lemma}
\label{l:varg} Consider matrix $M_{n}$ of (\ref{Mnm}), where $\{\mathbf{y}%
_{\alpha }\}_{\alpha =1}^{m_{n}}$ are i.i.d. random isotropic vectors having
an unconditional distribution and satisfying (\ref{a22L}), and $\{\tau
_{\alpha }\}_{\alpha =1}^{m_{n}}$ are non-negative
numbers satisfying (\ref{sigma}) and (\ref{m4}). Then we have
\begin{equation}
\mathbf{Var}\{\gamma _{n}(z)\}\leq C|\Im z|^{-6},  \label{var1}
\end{equation}%
and also $\forall \varepsilon \in (0,1/2]$
\begin{equation}
\mathbf{Var}\{\gamma _{n}(z)\}\leq \frac{C}{n|\Im z|^{4}}\sum_{\alpha
=1}^{m}(\tau _{\alpha }(|\Im z|+{\tau _{\alpha }}))^{3/2+\varepsilon }%
\mathbf{E}\{n^{-1}\Tr |G^{\alpha }(z)|^{2(1/2+\varepsilon) }\}
,  \label{vsum}
\end{equation}%
where $C$ does not depend on $n$ and $z$, $|G^{\alpha }|^2=G^{\alpha }G^{\alpha *}$, and
\begin{equation}
G^{\alpha }(z)=G(z)|_{\tau _{\alpha }=0}.  \label{gal}
\end{equation}%
If $\{\mathbf{y}_{\alpha }\}_{\alpha =1}^{m_{n}}$ are i.i.d. very good
vectors in the sense of Definition \ref{d:very}, then additionally
\begin{equation}
\mathbf{E}\{|\gamma _{n}(z)^{\circ }|^{4}\}\leq C|z|^{4}|\Im z|^{-12},
\label{g4<}
\end{equation}
\end{lemma}
\begin{proof}
Given an integer $\alpha \in [1,m]$, denote $\mathbf{E}_{\le \alpha}$ and
$\mathbf{E}_{\alpha}$  the expectation with respect to
$\{\mathbf{y}_{1 },..., \mathbf{y}_{\alpha }\}$ and $\mathbf{y}_{\alpha } $,
so that for any random variable $\xi$, depending on
$\{\mathbf{y}_{\alpha}\}_{\alpha=1}^m$ we have
$
\mathbf{E}_{\le 0}=\xi, \quad \mathbf{E}_{\le m}=\mathbf{E}\{\xi\}
$
and
\[
\xi -\mathbf{E}\{\xi\}=\sum_{\alpha=1}^m (\mathbf{E}_{\le \alpha-1}\{\xi\}-\mathbf{E}_{\le \alpha}\{\xi\}).
\]
By using the definition of $\mathbf{E}_{\le \alpha}$ and the above identity it is easy to find that
 \begin{equation}
 \mathbf{Var}\{\xi\}=\sum_{\alpha=1}^m\mathbf{E}\{|\mathbf{E}_{\le \alpha -1}\{\xi\}-
\mathbf{E}_{\le \alpha}\{\xi\}|^2\}.
\label{varxi}
\end{equation}
Denote also
\begin{equation}
M^\alpha_{n}=M_{n}|_{\tau_\alpha=0},\quad \gamma_n^\alpha(z) =\Tr G^\alpha(z), \quad \xi^\circ_\alpha=\xi-\mathbf{E}_{\alpha}\{\xi\}, \label{Ga}
\end{equation}
where $G^{\alpha}(z)$  is defined in (\ref{gal}).
Applying (\ref{varxi}) to $\xi=\gamma_n$ (see (\ref{gamma})) and using the
Schwarz inequality, we get
\begin{equation}
 \mathbf{Var}\{\gamma_n\}\le\sum_{\alpha=1}^m\mathbf{E}\{|(\gamma_n)^\circ_\alpha|^2\}.
\label{var=}
\end{equation}
Furthermore, since $M_{n}-M_{n}^{\alpha}=\tau_{\alpha}L_{\alpha}$ is
the rank one matrix (see (\ref{Mnm}) and (\ref{Ga})), we can write
the formula
 \begin{equation}
 G-G^\alpha=-\frac{\tau_\alpha G^\alpha L_\alpha G^\alpha}{1+\tau_\alpha (G^\alpha \mathbf{y}_{\alpha },\mathbf{y}_{\alpha })},\label{r1}
\end{equation}
implying for $\gamma_n$ and $\gamma_n^{\alpha}$ of (\ref{gal}) and (\ref{Ga})
 \begin{equation}
 \gamma _{n}-\gamma_{n}^\alpha =-\frac{\tau_\alpha ((G^{\alpha})^2 \mathbf{y}_{\alpha },\mathbf{y}_{\alpha })}{1+\tau_\alpha (G^\alpha \mathbf{y}_{\alpha },\mathbf{y}_{\alpha })}=-\frac{B_{\alpha n}}{A_{\alpha n}},\label{g-ga}
\end{equation}
where
\begin{align}
&A_{\alpha n} =1+\tau_\alpha (G^\alpha \mathbf{y}_{\alpha },\mathbf{y}_{\alpha }),\label{Aa}
\\
&B_{\alpha n}=\frac{d}{dz}A_{\alpha n}=\tau_\alpha ((G^{\alpha})^2 \mathbf{y}_{\alpha },\mathbf{y}_{\alpha }).\label{Ba}
\end{align}
It follows from the spectral theorem for real symmetric matrices that there
exists a non-negative measure $m^{\alpha}$ such that
\begin{equation}
 (G^\alpha \mathbf{y}_{\alpha },\mathbf{y}_{\alpha })= \int_0^\infty\frac{m^\alpha(d\lambda)}{\lambda-z}.
\label{gya}
\end{equation}
This and (\ref{Aa}) -- (\ref{Ba}) yield
\begin{equation}
|A_\alpha| \ge |\Im A_\alpha|=
\tau_\alpha |\Im (G^\alpha \mathbf{y}_{\alpha },\mathbf{y}_{\alpha }))|=\tau_\alpha |\Im z| \int_0^
\infty\frac{ m^\alpha(d\lambda)}{|\lambda-z|^{2}},
\label{Abou}
\end{equation}
and
\begin{equation}
|B_\alpha| \le
|\tau_\alpha| \int_0^
\infty\frac{ m^\alpha(d\lambda)}{|\lambda-z|^{2}},
\label{Bbou}
\end{equation}
implying
\begin{equation}
|B_{\alpha n}/A_{\alpha n}|\le 1/|\Im z|.\label{B/A<}
\end{equation}
This and the identity
\begin{equation}
\frac{1}{A}=\frac{1}{\mathbf{E}\{A\}}-\frac{A^\circ}{A\mathbf{E}\{A\}}, \;
A^\circ=A-\mathbf{E}\{A\}, \label{AEA}
\end{equation}
allows us to write
 \begin{align}
\mathbf{E}\{|(\gamma_n)^\circ_\alpha|^2\}&= \mathbf{E}\{|\gamma _{n}-\gamma^\alpha_{n}-\mathbf{E}_\alpha\{\gamma _{n}-\gamma^\alpha_{n}\}|^2\}\notag
\\
&\leq\mathbf{E}\Big\{\Big|\frac{B_{\alpha n}}{A_{\alpha n}}-
\frac{\mathbf{E}_\alpha\{B_{\alpha n}\}}{\mathbf{E}_\alpha\{A_{\alpha n}\}}\Big|^2\Big\}=\mathbf{E}\Big\{\Big|\frac{(B_{\alpha n})^\circ_\alpha}{\mathbf{E}_\alpha\{A_{\alpha n}\}}-\frac{B_{\alpha n}}{A_{\alpha n}}
\cdot\frac{(A_{\alpha n})^\circ_\alpha}{\mathbf{E}_\alpha\{A_{\alpha n}\}}
\Big|^2\Big\}\notag
\\
&\le
 2 \mathbf{E}\Big\{\Big|\frac{(B_{\alpha n})^\circ_\alpha}{\mathbf{E}_\alpha\{A_{\alpha n}\}}
\Big|^2\Big\}+\frac{2}{|\Im z|^{2}} \mathbf{E}\Big\{\Big|\frac{(A_{\alpha n})^\circ_\alpha}{\mathbf{E}_\alpha\{A_{\alpha n}\}}
\Big|^2\Big\}.\label{gAB}
\end{align}

Let us estimate the second term on the r.h.s. of (\ref{gAB}). Since  $(A_{\alpha n})^\circ_\alpha= $ \\
$\tau_\alpha (G^\alpha \mathbf{y}_{\alpha },\mathbf{y}_{\alpha })^\circ_\alpha$, then in view of definitions of (\ref{a22L})
 \begin{align}
\frac{1}{\tau_\alpha^{2}}\mathbf{E}_\alpha\{|(A_{\alpha n})^\circ_\alpha|^{2}\}
=\big(a_{2,2}-\frac{1}{n^{2}}\big)|\Tr G^\alpha|^2+2a_{ 2,2}\Tr |G^\alpha|^2 +\kappa_{4}\sum_{j=1}^n |G^\alpha_{jj}|^2.\label{VarGa}
\end{align}
 This, (\ref{a22L}) and the bound $||G^\alpha||\le|\Im z|^{-1} $ yield
\begin{align}
&\mathbf{E}\{|(A_{\alpha n})^\circ_\alpha|^{2}\}
\le \frac{C\tau_\alpha^{2}}{n^{2}}\Tr |G^\alpha|^2,\label{VarA}
\\
&\mathbf{E}\{|(B_{\alpha n})^\circ_\alpha|^{2}\}
\le \frac{C\tau_\alpha^{2}}{n^{2}}\Tr |G^\alpha|^4\le \frac{C\tau_\alpha^{2}}{n^{2}|\Im z|^{2}}\Tr |G^\alpha|^2.
\label{VarB}
\end{align}
It follows then from (\ref{var=}), (\ref{gAB}), and (\ref{VarA}) -- (\ref{VarB}) that
 \begin{equation}
\mathbf{Var}\{\gamma_n(z)\}\leq \frac{C}{n|\Im z|^2}\sum_{\alpha=1}^m\tau_\alpha^2
\mathbf{E}\Big\{\frac{n^{-1}\Tr|G^\alpha(z)|^{2}}{|\mathbf{E}_{\alpha}\{A_{\alpha n}\}|^{2}}\Big\}.\label{var<}
\end{equation}
Let $N_n^\alpha$ be the normalized counting measure of
$M_{n}^\alpha$. Then we have in view of (\ref{Ga}), (cf. (\ref{gya}))
\begin{equation}\label{gan}
\gamma_n^\alpha(z)=\int_0^\infty\frac{N_n^\alpha(d\lambda)}{\lambda-z}.
\end{equation}
In addition,  (\ref{isotropic}) and (\ref{Ga}) imply
\begin{equation}
 \mathbf{E}_\alpha\{A_{\alpha n}\}=1+\tau_\alpha n^{-1}\gamma_{n}^\alpha.
\label{EA}
\end{equation}
This and an argument similar to that leading to (\ref{B/A<}) yield
\begin{equation}
 \frac{n^{-1}\Tr |G^\alpha|^2}{|\mathbf{E}_\alpha\{A_{\alpha n}\}|}\le \frac{1}{\tau_\alpha|\Im z|}.
\label{1}
\end{equation}
It also follows from (\ref{r1}) that
\begin{equation}
A^{-1}_{\alpha n}=1+\tau_{\alpha}(G \mathbf{y}_{\alpha },\mathbf{y}_{\alpha})
\label{A=}
\end{equation}
implying, together with (\ref{isotropic}) and the Jensen inequality
\begin{equation}
{|\mathbf{E}_\alpha\{A_{\alpha n}\}|}^{-1}\le\mathbf{E}_\alpha\{|A_{\alpha n}|^{-1}\}\le  {1+\tau_{\alpha}}{|\Im z|^{-1}}.
\label{EA>>}
\end{equation}
Now (\ref{var<}), (\ref{EA>>}), and (\ref{m4}) yield (\ref{var1}).

Applying (\ref{1}),  (\ref{EA>>})  and Jensen inequality again, we get
\begin{align*}
\frac{n^{-1}\Tr|G^\alpha(z)|^{2}}{|\mathbf{E}_{\alpha}\{A_{\alpha n}\}|^{2}}&=\frac{1}{|\mathbf{E}_{\alpha}\{A_{\alpha n}\}|^{3/2+\varepsilon}}\Big(\frac{\Tr|G^\alpha(z)|^{2}}{n|\mathbf{E}_{\alpha}\{A_{\alpha n}\}|}\Big)^{1/2-\varepsilon}(n^{-1}\Tr|G^\alpha(z)|^{2})^{1/2+\varepsilon}
\\
&\le(  {1+\tau_{\alpha}}{|\Im z|^{-1}})^{3/2+\varepsilon}{(\tau_{\alpha}}{|\Im z|^{}})^{-1/2+\varepsilon}\cdot n^{-1}\Tr|G^\alpha(z)|^{2(1/2+\varepsilon)},
\end{align*}
This and  (\ref{var<}) lead to (\ref{vsum}).

To prove (\ref{g4<}) we will use an analog of (\ref{varxi}) for the 4th moment of  $\gamma_n^\circ=\gamma_n-\mathbf{E}\{\gamma_n\}$ (se e.g. \cite{Dh-Co:68} and \cite{Pa-Sh:11}, Section 18.1.2), which together with an argument analogous to that leading to (\ref{gAB}), yields
\begin{align}
 \mathbf{E}\{|\gamma_n^\circ|^4\}&\le m\sum_{\alpha=1}^m\mathbf{E}\{|(\gamma_n)^\circ_\alpha|^4\}\notag
 \\
 &\le Cm\sum_{\alpha=1}^m \mathbf{E}\Big\{\Big|\frac{(B_{\alpha n})^\circ_\alpha}{\mathbf{E}_\alpha\{A_{\alpha n}\}}
\Big|^4\Big\}+\frac{1}{|\Im z|^{4}} \mathbf{E}\Big\{\Big|\frac{(A_{\alpha n})^\circ_\alpha}{\mathbf{E}_\alpha\{A_{\alpha n}\}}
\Big|^4\Big\},
\label{4}
\end{align}
where (see (\ref{m4Ayy}))
\begin{equation}
 \mathbf{E}_\alpha\{|(A_{\alpha n})^\circ_\alpha|^4\}\le \tau_\alpha^4|/\Im z|^{4}n^{2},\quad \mathbf{E}_\alpha\{|(B_{\alpha n})^\circ_\alpha|^4\}\le \tau_\alpha^4|/\Im z|^{8}n^{2}.
\label{AB4<}
\end{equation}
Since the matrix $M_n$ with non-negative $\tau_\alpha, \; \alpha=1,...,m$
is positive definite, it follows from (\ref{gya}) that
\begin{equation*}
\Im (z (G^\alpha \mathbf{y}_{\alpha },\mathbf{y}_{\alpha }))=\Im \int_0^\infty\frac{\lambda m^\alpha(d\lambda)}{\lambda-z}=\Im z \cdot\int_0^\infty\frac{\lambda m^\alpha(d\lambda)}{|\lambda-z|^{2}}.
\end{equation*}
This yields the inequality $\Im z \cdot\Im (z (G^\alpha \mathbf{y}_{\alpha },\mathbf{y}_{\alpha}))\ge0$, so that
\begin{equation}
{|A_{\alpha n}|^{-1}}\le \Big|\frac{z}{\Im z +\tau_{\alpha}\Im (z (G^\alpha \mathbf{y}_{\alpha },\mathbf{y}_{\alpha }))}\Big|\le {|z|}{|\Im z|^{-1}},
\label{A>}
\end{equation}
and by the Jensen inequality
\begin{equation}
{|\mathbf{E}_\alpha\{A_{\alpha n}\}|}^{-1}\le  {|z|}{|\Im z|^{-1}}.
\label{EA>}
\end{equation}Now (\ref{g4<}) follows from (\ref{4}) -- (\ref{AB4<})
and (\ref{EA>}).
\end{proof}

\bigskip \noindent \textbf{Proof of Lemma \ref{l:Var}.} It follows from (\ref%
{varNf}) and (\ref{vsum}) that
\begin{align}
\mathbf{Var}\{\mathcal{N}_{n}[\varphi ]\}\leq C_{s}^{\prime }||\varphi
||_{s}^{2}\int_{0}^{\infty }& d\eta e^{-\eta }\eta ^{2s-5}\frac{1}{n}%
\sum_{\alpha =1}^{m}(\tau _{\alpha }(\eta +{\tau _{\alpha }}%
))^{3/2+\varepsilon }  \notag \\
& \times \int_{-\infty }^{\infty }\mathbf{E}\{n^{-1}\Tr|(G^{\alpha
}(z)|^{1+2\varepsilon }\}d\lambda ,  \label{vv}
\end{align}%
and we have for $z=\lambda+i\eta$ (cf. (\ref{gan}))
\begin{equation*}
\int_{-\infty }^{\infty }\mathbf{E}\{n^{-1}Tr|G^{\alpha
}(z)|^{1+2\varepsilon )}\}d\lambda =\int_{0}^{\infty }\mathbf{E}%
\{N_{n}^{\alpha }(d\lambda )\}\int_{-\infty }^{\infty }\frac{d\mu }{%
((\lambda -\mu )^{2}+\eta ^{2})^{1/2+\varepsilon }}\leq \frac{C}{\eta
^{2\varepsilon }}.
\end{equation*}%
Thus, for any $s=2+\delta $ $\delta >\varepsilon$, the integral over $\eta$
in (\ref{vv}) converges. Lemma \ref{l:Var} is proved.

\section{Proof of Theorem \protect\ref{t:clt}}

\label{s:main}
It suffices to show that if
\begin{equation}  \label{Z_n}
Z_n(x)=\mathbf{E}\{e_{n}(x)\},\quad e_{n}(x)=e^{ix\mathcal{N}_n^\circ[\varphi%
]},
\end{equation}%
then we have uniformly in $|x|\le C$
\begin{equation}  \label{limZ}
\lim_{n\rightarrow\infty}Z_n(x)=\exp\{-x^2V[\varphi]/2\}
\end{equation}
with $V[\varphi]$ of (\ref{Var}).
Define for any test-functions $\varphi\in H_{2+\delta}$
\begin{equation}  \label{phi_y}
\varphi_y=P_y*\varphi,
\end{equation}
where $P_y$ is the Poisson kernel
\begin{equation}  \label{P_y}
P_y(x)=\frac{y}{\pi(x^2+y^2)},
\end{equation}
and "$*$" denotes the convolution. We have
\begin{equation}  \label{appr}
\lim_{y\downarrow 0}||\varphi-\varphi_y||_{2+\delta}=0.
\end{equation}
Denote for the moment the characteristic function (\ref{Z_n}) by $%
Z_n[\varphi]$, to make explicit its dependence on the test function. We have then for any
converging subsequence $\{Z_{n_j}[\varphi]\}_{j=1}^\infty$
\begin{equation*}
\lim_{{n_j}\rightarrow\infty}Z_{n_j}[\varphi]=\lim_{y\downarrow 0}\lim_{{n_j}%
\rightarrow\infty}(Z_{n_j}[\varphi]-Z_{n_j}[\varphi_{y}])+\lim_{y\downarrow
0}\lim_{{n_j}\rightarrow\infty}Z_{n_j}[\varphi_{y}].
\end{equation*}
Since by (\ref{apriory}) and (\ref{appr})
\begin{equation*}
|Z_{n_j}[\varphi]-Z_{n_j}[\varphi_{y}]|\le |x|\mathbf{Var}\{\mathcal{N}%
_{n_j}[\varphi] -\mathcal{N}_{n_j}[\varphi_y]\}^{1/2}\le
C|x|||\varphi-\varphi_y||_{2+\delta}{\rightarrow}0,
\end{equation*}
as $y\downarrow 0$, then
\begin{equation}  \label{limny}
\lim_{{n_j}\rightarrow\infty}Z_{n_j}[\varphi]=\lim_{y\downarrow 0}\lim_{{n_j}%
\rightarrow\infty}Z_{n_j}[\varphi_{y}].
\end{equation}
Hence it suffices to find the limit of $Z_{yn}:=Z_{n}[\varphi_{y}]=\mathbf{%
E}\{e_{y,n}(x)\}$ with $e_{y,n}(x)=e^{ix\mathcal{N}_n^\circ[\varphi_y]}$, as ${n}%
\rightarrow\infty$.

It follows from (\ref{phi_y}) -- (\ref{P_y}) that
\begin{equation}  \label{repr_N}
\mathcal{N}_n[\varphi_ y]=\frac{1}{\pi}\int
\varphi(\mu)\Im\gamma_n(z)d\mu,\quad z=\mu+iy.
\end{equation}
This allows us to write
\begin{equation}
\frac{d}{dx}Z_{yn}(x)=\frac{1}{2\pi}\int \varphi(\mu)(
Y_{n}(z,x)- Y_{n}(\overline z,x))d\mu,  \label{dZ=}
\end{equation}
where 
$Y_n(z,x)=\mathbf{E}\{\gamma_n(z)e_{yn}^\circ(x)\}$. 
Now the first bound in (\ref{g4<}) yields
\begin{equation*}
|Y_n(z,x)|\le2\mathbf{Var}\{\gamma_n(z)\}^{1/2}\le C|\Im z|^{-6}.
\end{equation*}
This and the dominated
convergence theorem imply that if $\varphi\in L^1$, then the limit of the integral in (\ref{dZ=}) as $n \to \infty $ can be obtained from that of $Y_n$ for any fixed non-real $z$.

We have from the resolvent identity and (\ref{gamma})
\begin{equation*}
\gamma_n(z)=-\frac{n}{z}+\frac{1}{z}\Tr M_{n}G(z)=\frac{-n+m}{z}-\frac{1}{z%
}\sum_{\alpha=1}^m A^{-1}_{\alpha n}(z),
\end{equation*}
where $A_{\alpha n}$ is defined in (\ref{Aa}). This implies
\begin{align}  \label{Y=}
Y_n(z,x)&=-\frac{1}{z}\sum_{\alpha=1}^m\mathbf{E}\{A^{-1}_{\alpha
n}(z)(e_{n}^{\alpha}(x))^{\circ}\} -\frac{1}{z}\sum_{\alpha=1}^m\mathbf{E}%
\{A^{-1}_{\alpha n}(z)(e_{n}(x)-e_n^{\alpha}(x))^\circ\}  \notag \\
&=:T_{1}^{(n)}+T_{2}^{(n)}.
\end{align}
Iterating (\ref{AEA}) three times, we get for $T_{1}^{(n)}$ of (\ref{Y=}):
\begin{align}  \label{T1=}
T_{1}^{(n)}=\frac{1}{z}\sum_{\alpha=1}^m&\frac{\mathbf{E}\{A^{\circ}_{\alpha
n}(z)(e_{yn}^{\alpha}(x))^{\circ}\}}{\mathbf{E}\{A_{\alpha n}(z)\}^{2}} -\frac{1}{%
z}\sum_{\alpha=1}^m\frac{\mathbf{E}\{(A^{\circ}_{\alpha
n}(z))^2(e_{yn}^{\alpha}(x))^{\circ}\}}{\mathbf{E}\{A_{\alpha
n}(z)\}^{3}}  \notag
\\
&+\frac{1}{z}\sum_{\alpha=1}^m\frac{\mathbf{E}\{(A^{\circ}_{\alpha
n}(z))^3A^{-1}_{\alpha
n}(z)(e_{yn}^{\alpha}(x))^{\circ}\}}{\mathbf{E}\{A_{\alpha
n}(z)\}^{3}}=:T_{11}^{(n)}-T_{12}^{(n)}+T_{13}^{(n)}.
\end{align}
It follows from (\ref{EA0p<}) and (\ref{A>}) -- (\ref{EA>}) that
\begin{align}  \label{T13=}
T_{13}^{(n)}=O(n^{-1/2}),\quad n\rightarrow\infty.
\end{align}
Consider  now $T_{11}^{(n)}$. Since $e_{yn}^{\alpha}$ does not depend on $\mathbf{%
y}_\alpha $,  (\ref{EA}) implies
\begin{align}
\mathbf{E}\{A_{\alpha n}(z)(e_{yn}^{\alpha}(x))^{\circ}\}&=\mathbf{E}\{\mathbf{E}%
_\alpha\{A_{\alpha n}(z)\}(e_{yn}^{\alpha}(x))^{\circ}\}  \notag \\
&=\tau_{\alpha}
n^{-1}\mathbf{E}\{\gamma^{\alpha}_n(z)(e_{yn}^{\alpha}(x))^{
\circ}\}=\tau_{\alpha}n^{-1}Y_n(z,x)+R_n,  \label{Eae=}
\end{align}
where $R_n=\tau_{\alpha}n^{-1}\mathbf{E}\{\gamma^{\alpha}_n
(e_{yn}^{\alpha})^{\circ}-\gamma_ne_{yn}^{\circ}\}$.
Applying consequently (\ref{repr_N}), the Schwarz inequality and then
(\ref{var1}), (\ref{vg-ga}) and (\ref{phi_y}), we get 
\begin{align}
\tau_{\alpha}^{-1}n|R_n|&\le2\mathbf{E}\{|( \gamma
_{n}-\gamma_{n}^{\alpha})^{\circ}(z) |\}+\frac{|x|}{\pi}\int |\varphi(\lambda_2)|\mathbf{E}\{| \gamma^{\circ} _{n} (z)||( \gamma
_{n}-\gamma_{n}^{\alpha})^{\circ}(z_2)|\} d\lambda_2  \notag \\
&\le C_z\tau_\alpha n^{-1/2}.  \label{Rn<}
\end{align}
Here and below we denote by $C_z$ any positive quantity depending
only on $|\Im z|$ and $|z|$. It follows then from (\ref{Eae=}) -- (\ref{Rn<})
that
\begin{align*}
\mathbf{E}\{A_{\alpha
n}(z)(e_{n}^{\alpha}(x))^{\circ}\}&=\tau_{\alpha}n^{-1}Y_n(z,x)+O(n^{-3/2}).
\end{align*}
Plugging this and (\ref{lim1/EA}) in $T_{11}^{(n)}$ of (\ref{T1=}), we get
\begin{align*}
T_{11}^{(n)}=Y_n(z,x)\frac{1}{nz}\sum_{\alpha=1}^m\frac{\tau_{\alpha}}{%
(1+\tau_\alpha f(z))^2}+o(1),\quad n\rightarrow\infty,
\end{align*}
and by (\ref{sigma}) and (\ref{nmc})
\begin{align}  \label{T11=}
T_{11}^{(n)}=Y_n(z,x)\frac{c}{z}\int_{0}^\infty\frac{\tau d\sigma(\tau)}{%
(1+\tau f(z))^2}+o(1),\quad n\rightarrow\infty.
\end{align}
Next, it follows from (\ref{A0Aa0}) -- (%
\ref{ga4<}), the Schwarz inequality and (\ref{varEA}) that
\begin{align}
|\mathbf{E}\{(A^{\circ}_{\alpha n})^2 (e_{yn}^{\alpha})^{\circ}\}|&= \mathbf{|E}\{%
\mathbf{E}_\alpha\{((A_{\alpha n})^{\circ
}_\alpha)^2\}(e_{yn}^{\alpha})^{\circ}\}
+\tau_{\alpha}^2n^{-2}\mathbf{E}\{((\gamma_n^{\alpha})^{\circ})^2
(e_{yn}^{\alpha})^{ \circ}\}|\notag
\\
&= o(n^{-1}),  \label{3}
\end{align}
hence,
\begin{align}  \label{T12=}
T_{12}^{(n)}=o(1), \quad n\rightarrow\infty.
\end{align}
Now (\ref{T13=}), (\ref{T11=}) and (\ref{T12=}) yield for $%
T_1^{(n)}$ of (\ref{T1=}):
\begin{align}  \label{T1}
T_{1}^{(n)}=Y_n(z,x)\frac{c}{z}\int_{0}^\infty\frac{\tau d\sigma(\tau)}{%
(1+\tau f(z))^2}+o(1),\quad n\rightarrow\infty.
\end{align}
Consider  $T_{2}^{(n)}$ of (\ref{Y=}). Since by (\ref{repr_N})
\begin{align*}
e_{yn}-e_{yn}^\alpha=&\frac{ixe_{yn}^\alpha}{\pi}\int _{  %
}\varphi(\lambda_1)\Im ( \gamma
_{n}-\gamma_{n}^{\alpha})^{\circ}(z_1)d\lambda_1 \\ 
&+O\Big(\Big|\int _{%
  }\varphi(\lambda)\Im ( \gamma
_{n}-\gamma_{n}^{\alpha})^{\circ}(z_1)d\lambda_1\Big|^2\Big),
\end{align*}
then in view of (\ref{g-ga})
\begin{align}
\mathbf{E}\{A_{\alpha n}^{-1}(z)(e_{yn}-e_{yn}^\alpha)^\circ(x)\}=&\frac{ix}{%
\pi}\int \varphi(\lambda_1) \mathbf{E}\{e_{yn}^\alpha(x)(A_{%
\alpha n}^{-1})^\circ(z) \Im (B_{\alpha n}A_{\alpha
n}^{-1})^{\circ}(z_1)\}d\lambda_1  \notag \\
&+\int \int
O(R_n^{(1)})\varphi(\lambda_1)%
\varphi(\lambda_2)d\lambda_1 d\lambda_2,  \label{e-e=}
\end{align}
where
\begin{align*}
R_n^{(1)}= \mathbf{E}\{(A_{\alpha n}^{-1})^\circ(z) \Im (B_{\alpha
n}A_{\alpha n}^{-1})^{\circ}(z_1) \Im (B_{\alpha n}A_{\alpha
n}^{-1})^{\circ}(z_\nu)\}.
\end{align*}
Applying (\ref{AEA}) to $A_{\alpha n}^{-1}(z) $, $A_{\alpha
n}^{-1}(z_1) $, $A_{\alpha n}^{-1}(z_2), $ we get
\begin{align*}
\mathbf{E}\{&(A_{\alpha n}^{-1})^\circ(z) (B_{\alpha n}A_{\alpha
n}^{-1})^{\circ}(z_1) (B_{\alpha n}A_{\alpha n}^{-1})^{\circ}(z_2)\}
\\
&=\frac{\mathbf{E}\{{(A_{\alpha n}^\circ A_{\alpha n}^{-1})^\circ(z)} {%
(B_{\alpha n}-A_{\alpha n}^\circ A_{\alpha n}^{-1}B_{\alpha
n})^\circ(z_1)} {(B_{\alpha n}-A_{\alpha n}^\circ A_{\alpha
n}^{-1}B_{\alpha n})^\circ(z_2)}\}}{\mathbf{E}\{A_{\alpha n}(z) \}\mathbf{E%
}\{A_{\alpha n}(z_1) \}{\mathbf{E}\{A_{\alpha n}(z_2) \}}}.
\end{align*}
This and (\ref{EA0p<}) yield
\begin{align}
R_n^{(1)}= O(n^{-1/2}),\quad n\rightarrow\infty.  \label{Rn1}
\end{align}
Similarly, iterating twice (\ref{AEA}), applying (\ref{EA0p<}) and (\ref%
{varEA}) and taking into account that only linear terms in $A_{\alpha n}^\circ$ and $%
B_{\alpha n}^\circ$ give non-vanishing contribution, we get
\begin{align*}
\mathbf{E}&\{e^\alpha_{yn}(x)(A_{\alpha n}^{-1})^\circ(z)(B_{\alpha
n}A_{\alpha n}^{-1})^\circ(z_1)\} \\
&=-\frac{\mathbf{E}\{e^\alpha_{yn}(x)A_{\alpha n}^\circ(z)B_{\alpha
n}^\circ(z_1)\}} {\mathbf{E}\{A_{\alpha n}(z) \}^2\mathbf{E}%
\{A_{\alpha n}(z_1) \}} +\frac{\mathbf{E}\{e^\alpha_{yn}(x)A_{\alpha
n}^\circ(z)A_{\alpha n}^\circ(z_1)B_{\alpha n}(z_1)\}}{\mathbf{E}%
\{A_{\alpha n}(z) \}^2\mathbf{E}\{A_{\alpha n}(z_1)
\}^2}+O(n^{-3/2})
\\
&=\mathbf{E}\{e^\alpha_{yn}(x)\}\Big[-\frac{\mathbf{E}\{A_{\alpha
n}^\circ(z)B_{\alpha n}^\circ(z_1)\}} {\mathbf{E}\{A_{\alpha n}(z) \}^2%
\mathbf{E}\{A_{\alpha n}(z_1) \}} \\
&\hspace{4.5cm}+\frac{\mathbf{E}\{A_{\alpha n}^\circ(z)A_{\alpha
n}^\circ(z_1)\}\mathbf{E}\{B_{\alpha n}(z_1)\}}{\mathbf{E}%
\{A_{\alpha n}(z) \}^2\mathbf{E}\{A_{\alpha n}(z_1) \}^2}\Big]%
+O(n^{-3/2}).
\end{align*}
This, the bound $|Z_{yn}-\mathbf{E}%
\{e^\alpha_{yn}\}|=O{(n^{-1/2}})$  (cf. (\ref{Rn<})), and (\ref{Ba}) imply
\begin{align}
\mathbf{E}\{e^\alpha_{yn}&(x)(A_{\alpha n}^{-1})^\circ(z)(B_{\alpha
n}A_{\alpha n}^{-1})^\circ(z_1)\}  \notag \\
&=-Z_{yn}(x)\frac{1} {\mathbf{E}\{A_{\alpha n}(z) \}^2}\frac{\partial}{%
\partial z_1} \frac{\mathbf{E}\{A_{\alpha n}^\circ(z)A_{\alpha
n}^\circ(z_1)\}} {\mathbf{E}\{A_{\alpha n}(z_1) \}}+O(n^{-3/2})
\notag \\
&=-Z_{yn}(x)\frac{\tau_\alpha^2D(z,z_1)} {(1+\tau_\alpha f(z))^2}%
+O(n^{-3/2}),  \label{EAB/A}
\end{align}
where
\begin{align}
D_{\tau_\alpha }(z,z_1)=\frac{\partial}{\partial z_1} \frac{{%
2\Delta f}/{\Delta z+(a+b)f(z)f(z_1)}}{1+\tau_\alpha f(z_1)},
\label{D}
\end{align}
and we used (\ref{lim1/EA}) and (\ref{limAA}). Plugging (\ref{Rn1}) -- (\ref{D})
in (\ref{e-e=}) and applying (\ref{sigma}) -- (\ref{nmc}), we get for $%
T_2^{(n)}$ of (\ref{Y=}):
\begin{align*}
T_2^{(n)} =-\frac{xZ_{yn}(x)}{2\pi z}\int d\lambda_1\varphi(%
\lambda_1) \;c\int_{0}^\infty\frac{\tau^2d\sigma (\tau)} {(1+\tau f(z))^2}%
[D_{\tau}(z,z_1)-D_{\tau}(z,\overline{z_1})]+o(1).
\end{align*}
This and (\ref{Y=}) -- (\ref{T1=}) yield
\begin{align*}
Y_{n}(z,x)=&\Big(c\int_{0}^\infty\frac{\tau d\sigma (\tau)} {(1+\tau f(z))^2}%
-z\Big)^{-1} \\
&\times\frac{xZ_{yn}(x)}{2\pi }\int
d\lambda_1\varphi(\lambda_1)
\;c\int_{0}^\infty\frac{\tau^2d\sigma (\tau)} {(1+\tau f(z))^2}%
[D_{\tau}(z,z_1)-D_{\tau}(z,\overline{z_1})]+o(1),
\end{align*}
and after some calculations based on (\ref{f}) we finally get
\begin{align}
Y_{n}(z,x)=\frac{xZ_{yn}(x)}{2\pi }\int   d\lambda_1\varphi(%
\lambda_1) \;[C(z,z_1)-C(z,\overline{z_1})]+o(1),  \label{Y}
\end{align}
where $C(z,z_1)$ is defined in (\ref{Cov}). Now it follows from (\ref%
{dZ=}) and (\ref{Y}) that
\begin{align}
\frac{d}{dx}Z_{yn}(x)={-xV_{y}[\varphi]Z_{yn}(x)}+o(1),  \label{dZ}
\end{align}
where $V_{y}[\varphi]$ is given in (\ref{Var}). If we consider $\widetilde{Z}%
_{yn}(x)=e^{x^2V_{y}[\varphi]}Z_{yn}(x)$, then (\ref{dZ}) yields
\begin{align*}
{d}\widetilde{Z}_{yn}(x)/dx=o(1),\quad n\rightarrow\infty
\end{align*}
for any $|x|\le C $, and since $\widetilde{Z}_{yn}(0)={Z}_{yn}(0)=1$, we
obtain $\widetilde{Z}_{yn}(x)=1+o(1)$ uniformly in $|x|\le C$. Hence,
\begin{align*}
\lim_{n\rightarrow\infty}Z_{yn}(x)=\exp\{{-x^2V_{y}[\varphi]}\}.
\end{align*}
Now we take into account (\ref{limny}), allowing us to pass to the limit
$y\downarrow 0$, and obtain  (\ref{Var}). The theorem is proved.

\section{Auxiliary results}

\label{s:aux}

\begin{lemma}
\label{l:aux} Under conditions of Theorem \ref{t:clt} we have:
\begin{align}
&(i)\;\lim_{n\rightarrow\infty}\mathbf{E}\{n^{-1}\gamma^{\alpha}_n\}=f,
\;|\gamma_n-\gamma_n^\alpha|=O(1),\quad n\rightarrow\infty,  \label{ga-f} \\
&(ii)\;\mathbf{Var}\{\gamma_n-\gamma_n^\alpha\}=O(n^{-1}),\quad
n\rightarrow\infty,  \label{vg-ga} \\
&(iii)\;\mathbf{E}\{|A_{\alpha n}^\circ|^p\}\le \frac{C\tau_\alpha^{p}}{%
n^{p/2}|\Im z|^{p}},\;\mathbf{E}\{|B_{\alpha n}^\circ|^p\}\le \frac{%
C\tau_\alpha^{p}}{n^{p/2}|\Im z|^{2p}},\quad p=1,2,3,4,  \label{EA0p<} \\
&(iv)\; {\mathbf{E}\{A_{\alpha n}\}^{-1}}=(1+\tau_\alpha f)^{-1}+o(1),\quad
n\rightarrow\infty,  \label{lim1/EA} \\
&(v)\;\mathbf{Var}\{n\mathbf{E}_\alpha\{A_{\alpha n}^\circ(z_1)A_{\alpha
n}(z_2)\}\}=o(1),\quad n\rightarrow\infty,  \label{varEA} \\
&(vi)\;\lim_{n\rightarrow\infty}\tau_\alpha^{-2}n\mathbf{E}\{A_{\alpha
n}^\circ(z_1)A_{\alpha n}(z_2)\}=(a+b)f(z_{1})f(z_2)+{2\Delta f}/{\Delta z},
\label{limAA}
\end{align}
where $\gamma_n^\alpha$, $A_{\alpha n}$, $B_{\alpha n}$ are defined in (\ref%
{Ga}) and (\ref{Aa}) -- (\ref{Ba}), and $f$ is a unique solution of (\ref{MPE}).
\end{lemma}

\begin{proof}
{\it (i)} It follows from (\ref{g-ga}) and (\ref{B/A<})
that $ |\gamma_n-\gamma_n^\alpha|\le |\Im z|^{-1}$. On the other hand,
Theorem \ref{t:MPP} implies $\lim_{n\rightarrow\infty}\mathbf{E}\{n^{-1}\gamma_n\}=f$.
This leads to (\ref{ga-f}).

\medskip
{\it (ii)}  Consider  $V=\mathbf{Var}\{\Delta \gamma _{n}\}$, $ \Delta \gamma _{n}=\gamma _{n}-\gamma_{n}^{\alpha}$. By  (\ref{g-ga})  and (\ref{AEA})
we have \begin{align*}
V=\mathbf{E}\{\Delta \gamma _{n}\Delta\overline{ \gamma} _{n}^{\circ}\}&=-\mathbf{E}\{(B_{\alpha n}/A_{\alpha n})\Delta\overline{ \gamma} _{n}^{\circ}\}
\\
&=-\mathbf{E}\{B_{\alpha n}\Delta\overline{ \gamma} _{n}^{\circ}\}/\mathbf{E}\{A_{\alpha n} \}-\mathbf{E}\{A_{\alpha n}^{\circ}(B_{\alpha n}/A_{\alpha n})\Delta\overline{ \gamma} _{n}^{\circ}\}/\mathbf{E}\{A_{\alpha n} \},
 \end{align*}
and by  (\ref{B/A<}), (\ref{EA>>}), the Schwarz  inequality, and
(\ref{VarA}) -- (\ref{VarB})
\begin{align*}
V\le CV^{1/2}(\mathbf{Var}\{B_{\alpha n}\}^{1/2}+\mathbf{Var}\{A_{\alpha n}\}^{1/2}/|\Im z|)\le CV^{1/2}n^{-1/2}.
 \end{align*}
This yields (\ref{vg-ga}).

\medskip
{\it (iii)} Note  that
\begin{equation}
A_{\alpha n}^\circ=(A_{\alpha n})_\alpha^\circ+\tau_{\alpha}n^{-1}(\gamma^{\alpha})^{\circ}_n,\label{A0Aa0}
\end{equation}
and that (\ref{var1}) and (\ref{g4<}) imply
\begin{equation}
\mathbf{Var}\{ \gamma _{n}^\alpha\},\;\mathbf{E}\{|(\gamma^{\alpha}_n)^\circ|^4\}=O(1),
\quad n\rightarrow\infty. \label{ga4<}
\end{equation}
This allows us to replace  $A_{\alpha n}^\circ$ by
$(A_{\alpha n})^\circ_\alpha$ as $n \to \infty$
(see e.g. (\ref{3})). In view of (\ref{VarA}) and (\ref{A0Aa0}) -- (\ref{ga4<})
we have
\begin{equation}
\mathbf{E}\{|A_{\alpha n}^\circ|^2\}=\mathbf{E}\{|(A_{\alpha n})_\alpha^\circ|^2\}+\tau^{2}_{\alpha}n^{-2}\mathbf{Var}\{ \gamma _{n}^\alpha\}=O(n^{-1}). \label{5}
\end{equation}
This yields (\ref{EA0p<}) for $p=2$.
 Similarly, (\ref{EA0p<}) for $p=3,4$ follows from (\ref{m4Ayy})
 and  (\ref{A0Aa0}) -- (\ref{ga4<}).

\medskip
{\it (iv)} We have ${\mathbf{E}\{A_{\alpha n}\}^{-1}}=(1+\tau_\alpha f)^{-1}+r_n$,
where
\begin{equation*}
r_n=\tau_\alpha(f-\mathbf{E}\{n^{-1}\gamma^{\alpha}_n\})
{\mathbf{E}\{A_{\alpha n}\}^{-1}}(1+\tau_\alpha f)^{-1}.
\end{equation*}
The bound $|(1+\tau_\alpha f)^{-1}|\le |z/\Im z|$, (\ref{EA>})  and
(\ref{ga-f}) imply $r_n=o(1),$ hence (\ref{lim1/EA}).

\medskip
{\it (v)} It follows from (\ref{yyyy}) --  (\ref{Var=}) and  (\ref{a22}) --  (\ref{k4}) that
  \begin{align}
 n\tau_{\alpha}^{-2}&\mathbf{E}_\alpha\{(A_{\alpha n})_\alpha^\circ(z_1)A_{\alpha n}(z_2)\}=n\Big[(a_{2,2}-n^{-2})\Tr G^\alpha(z_{1})\Tr G^\alpha(z_{2})\label{EAA}
 \\
 &\hspace{2cm}+2a_{2,2}\Tr  G^\alpha(z_{1})G^\alpha(z_{2})+\kappa_{4}\sum_{j=1}^n G^\alpha_{jj}(z_{1})
 G^\alpha_{jj}(z_{2})\Big]\notag
\\
 &=an^{-2}\gamma^{\alpha}_n(z_1)\gamma^{\alpha}_n(z_2)+2n^{-1}\frac{\gamma^{\alpha}_n(z_1)
 -\gamma^{\alpha}_n(z_2)}{z_{1}-z_{2}}+bg_n^\alpha(z_{1},z_{2})+O(n^{-1}),\notag
 \end{align}
where $g_n^\alpha$ is defined in (\ref{gna}) (see Lemma \ref{l:g12}). Now (\ref{varEA}) follows from  (\ref{ga4<}) -- (\ref{5}) and  (\ref{vg12}).

\medskip
{\it (vi)}  The relation (\ref{limAA}) follows from (\ref{ga-f}), (\ref{EAA}) and  (\ref{gff}).
\end{proof}

\begin{lemma}
\label{l:g12} Consider
\begin{equation}\label{gn12}
g_n(z_1,z_2)=n^{-1}\sum_{j=1}^nG_{jj}(z_1)G_{jj}(z_2).
\end{equation}
Then under conditions of Theorem \ref{t:clt} we have as $%
n\rightarrow\infty$
\begin{align}
&\mathbf{E}\{ g_n(z_1,z_2)\}=f(z_1)f(z_2)+o(1),  \label{gff} \\
&\mathbf{Var}\{ g_n(z_1,z_2)\}=o(1),  \label{vg12}
\end{align}
where $f$ is given by (\ref{MPE}).
\end{lemma}

\begin{proof}
 Since by the resolvent identity
 \begin{align*}
G_{jj}(z_1)=-\frac{1}{z_{1}}+\frac{1}{z_{1}}(MG)_{jj}(z_1)=-\frac{1}{z_{1}}+\frac{1}{z_{1}}
\sum_{\alpha=1}^{m}\tau_{\alpha}y_{\alpha j}(G\mathbf{y}_{\alpha })_j(z_1),
 \end{align*}
 then
  \begin{align}
\mathbf{E}\{ g_n(z_1,z_2)\}=-\frac{\mathbf{E}\{n^{-1} \gamma_ n(z_2)\}}{z_{1}}+
\sum_{\alpha,j=1}^{m,n}\frac{\tau_{\alpha}\mathbf{E}\{y_{\alpha j}(G\mathbf{y}_{\alpha })_j(z_1)G_{jj}(z_2)\}}{nz_{1}}.\label{Eg}
 \end{align}
It follows from (\ref{r1}) that $(G\mathbf{y}_{\alpha })_j=A_{\alpha n}^{-1}(G^{\alpha}\mathbf{y}_{\alpha })_j$, $G_{jj}=G_{jj}^\alpha-\tau_{\alpha}A_{\alpha n}^{-1}(G^{\alpha}\mathbf{y}_{\alpha })_j^{2}$,
 where $A_{\alpha n}$ is defined in (\ref{Aa}).   Applying (\ref{AEA}), we
 obtain
 \begin{align}
\mathbf{E}\{y_{\alpha j}(G\mathbf{y}_{\alpha })_j(z_1)G_{jj}(z_2)\}&=\mathbf{E}\{y_{\alpha j}(A_{\alpha n}^{-1}(G^{\alpha}\mathbf{y}_{\alpha })_j)(z_1)G^{\alpha}_{jj}(z_2)\}\label{E}
\\
&\quad-\tau_{\alpha}\mathbf{E}\{y_{\alpha j}(A_{\alpha n}^{-1}(G^{\alpha}\mathbf{y}_{\alpha })_j)(z_1)(A_{\alpha n}^{-1}(G^{\alpha}\mathbf{y}_{\alpha })_j^{2})(z_2)\}\notag
\\
&=\mathbf{E}\{A_{\alpha n}(z_{1})\}^{-1}\mathbf{E}\{\mathbf{E}_\alpha\{y_{\alpha j}(G^{\alpha}\mathbf{y}_{\alpha })_j(z_1)\}G^{\alpha}_{jj}(z_2)\}+r_{n},\notag
 \end{align}
 where
 \begin{align}
r_n=&\frac{\mathbf{E}\{y_{\alpha j}(A^{\circ}_{\alpha n}A_{\alpha n}^{-1}(G^{\alpha}\mathbf{y}_{\alpha })_j)(z_1)G^{\alpha}_{jj}(z_2)\}}{\mathbf{E}\{A_{\alpha n}(z_{1})\}}\label{rn=}
\\
&-\frac{\mathbf{E}\{y_{\alpha j}((1+A^{\circ}_{\alpha n}A_{\alpha n}^{-1})(G^{\alpha}\mathbf{y}_{\alpha })_j)(z_1)((1+A^{\circ}_{\alpha n}A_{\alpha n}^{-1})^{-1}(G^{\alpha}\mathbf{y}_{\alpha })_j^{2})(z_2)\}}{\mathbf{E}\{A_{\alpha n}(z_{1})\}\mathbf{E}\{A_{\alpha n}(z_{2})\}}.\notag
 \end{align}
Since $|(G^{\alpha}\mathbf{y}_{\alpha })_j|\le|\Im z|^{-2}||\mathbf{y}_{\alpha }||^2$, then by the Schwarz inequality, (\ref{EA0p<}), and (\ref{a22}) the first term in
the r.h.s. of (\ref{rn=}) is less than
$$C\mathbf{Var}\{A_{\alpha n}\}^{1/2}\mathbf{E}\{y_{\alpha j}^2|(G^{\alpha}\mathbf{y}_{\alpha })_j|^{2}\}=O(n^{-3/2}).
$$
 By the same reason, all the terms in the numerator of the second term on the r.h.s. of (\ref{rn=}), which contain $A^{\circ}_{\alpha n}$ are of the order $O(n^{-3/2})$. For the only term that does not contain $A^{\circ}_{\alpha n}$ we have
 \begin{align*}
\mathbf{E}\{y_{\alpha j}(G^{\alpha}\mathbf{y}_{\alpha })_j(z_1)(G^{\alpha}\mathbf{y}_{\alpha })_j^{2}(z_2)\}&=\sum_{p,q,s=1}^n G_{jp}^\alpha (z_1)G_{jq}^\alpha  (z_2)G_{js}^\alpha(z_2)
\mathbf{E}\{y_{\alpha j}y_{\alpha p}y_{\alpha q}y_{\alpha s}\}
\\
&=3a_{2,2}G^\alpha_{jj}(G^\alpha G^\alpha)_{jj}+\kappa_{4} G_{jj}^\alpha G_{jj}^\alpha G_{jj}^\alpha=O(n^{-2}),
 \end{align*}
where we used (\ref{yyyy}) and (\ref{a22L}).
 Hence
 \begin{equation}
 r_n=O(n^{-3/2}),\quad n\rightarrow\infty.\label{rn}
 \end{equation}
 Besides, it follows from (\ref{isotropic}) that
 \begin{equation}
 \mathbf{E}_\alpha\{y_{\alpha j}(G^{\alpha}(z_{1})\mathbf{y}_{\alpha })_j\}=n^{-1}G_{jj}^{\alpha}(z_{1}).
 \label{2}
 \end{equation}
 Plugging (\ref{rn}) -- (\ref{2}) in (\ref{E}) and then  in (\ref{Eg}), we
 get
  \begin{align*}
\mathbf{E}\{ g_n(z_1,z_2)\}=-\frac{1}{z_{1}}f(z_2)+\frac{1}{n^2z_{1}}
\sum_{\alpha,j=1}^{m,n}\frac{\tau_{\alpha}\mathbf{E}\{ g_n^\alpha(z_1,z_2)\}}{1+\tau_{\alpha}f(z_1)}+o(1),\quad n\rightarrow\infty,
 \end{align*}
 where
  \begin{equation}
 g_n^\alpha(z_1,z_2)=n^{-1}\sum_{j=1}^nG^\alpha_{jj}(z_1)G^\alpha_{jj}(z_2),\label{gna}
 \end{equation}
and we took into account (\ref{ga-f}), (\ref{lim1/EA}). Using the
argument similar to that leading to (\ref{rn}), we obtain
$\mathbf{E}\{ g_n-g_n^\alpha\}=o(1)$, $n\rightarrow\infty$.
 This and (\ref{sigma}) show that
  \begin{align*}
\lim_{n\rightarrow\infty}\mathbf{E}\{ g_n(z_1,z_2)\}=\Big(c\int&\frac{\tau
d\sigma (\tau)}
{1+\tau f(z_{1})}-z_{1}\Big)^{-1}f(z_{2})=f(z_{1})f(z_{2}),
 \end{align*}
 and we get (\ref{gff}). The proof of (\ref{vg12}) follows the scheme of proof
 in Lemma \ref{l:varg}.
 \end{proof}

\end{document}